\documentclass[11pt]{amsart}
\usepackage{amsmath,amssymb,amsfonts,amsthm,amsopn}
\usepackage{latexsym,graphicx}
\usepackage{pb-diagram}
\newcommand{\tfa}{time-frequency analysis}

\newcommand{\stft}{short-time Fourier transform}

\newcommand{\modsp}{modulation space}

\newtheorem{theorem}{Theorem}[section]
\newtheorem{lemma}[theorem]{Lemma}

\newtheorem{proposition}[theorem]{Proposition}
\newtheorem{definition}[theorem]{Definition}

\newtheorem{remark}[theorem]{Remark}

\newcommand{\beqa}{\begin{eqnarray*}}
\newcommand{\eeqa}{\end{eqnarray*}}

\newcommand{\field}[1]{\mathbb{#1}}
\newcommand{\bR}{\field{R}}        %  real numbers
\newcommand{\bN}{\field{N}}        %  natural numbers
\newcommand{\bZ}{\field{Z}}        %  whole numbers
\def\G{\mathcal{G}}

\def\la{\lambda}

\def\eps{\epsilon}

\def\cS{\mathcal{S}}

\def\cG{\mathcal{G}}
\def\cM{\mathcal{M}}

\def\cC{\mathcal{C}}

\def\cN{\mathcal{N}}

\def\a{\aleph}

\def\rd{\bR^d}

\def\rdd{{\bR^{2d}}}
\def\zdd{{\bZ^{2d}}}

\def\lrd{L^2(\rd)}

\def\intrd{\int_{\rd}}

\def\R{\right)}

\def\<{\left<}
\def\>{\right>}

\def\mv1{M_v^1}

\def\phas{(x,\o )}
\def\mn{(m,n)}
\def\mn'{(m',n')}

\def\o{\xi}
\def\a{\alpha}
\def\b{\beta}

\def\R{\mathbb{R}}
\def\Ren{\mathbb{R}^d}
\def\Renn{\mathbb{R}^{2d}}

\def\sch{\mathcal{S}}

\def\Fur{\mathcal{F}}

\def\Sn2{S_{2}(L^{2}(\Ren))}
\def\S1{S_{1}(L^{2}(\Ren))}
\def\sig00{\sigma_{0,0}}
\def\la{\langle}
\def\ra{\rangle}

\begin{document}
\begin{abstract}
We give a brief survey of recent results concerning almost diagonalization of pseudodifferential operators via Gabor frames. Moreover,  we show new connections between symbols  with Gevrey, analytic or ultra-analityc regularity and time-frequency analysis of the corresponding pseudodifferential operators. 
\end{abstract}

\title[A few remarks on time-frequency analysis]{A few remarks on time-frequency analysis of Gevrey, analytic and ultra-analytic functions}

\author{Elena Cordero,  Fabio Nicola  and Luigi Rodino}
\address{Dipartimento di Matematica,
Universit\`a di Torino, via Carlo Alberto 10, 10123 Torino, Italy}
\address{Dipartimento di Scienze Matematiche,
Politecnico di Torino, corso Duca degli Abruzzi 24, 10129 Torino,
Italy}
\address{Dipartimento di Matematica,
Universit\`a di Torino, via Carlo Alberto 10, 10123 Torino, Italy}

\email{elena.cordero@unito.it}
\email{fabio.nicola@polito.it}
\email{luigi.rodino@unito.it}
%\thanks{}

%\date{}
\maketitle
\section{Introduction}
The objective of this paper is to report on recent progress concerning time-frequency analysis of pseudodifferential operators with smooth symbols. We discuss in particular connections between, on one side, symbols in the classical Gevrey or (ultra-)analytic spaces, and, on the other, (ultra-)modulation spaces,  Gabor frames, which are the framework of time-frequency analysis.\par

Modulation spaces having moderate weights were introduced by Feichtinger in 1983, and since then have been extensively studied by Feichtinger and many other authors, cf. in particular the textbook \cite{grochenig}, to which we address for further references. We recall their definition and main properties in the subsequent Section \ref{moddef}, but the main insight to understand their introduction  being to consider the
decay property of a function with respect to the space variable and the variable of
its Fourier transform simultaneously.
The time-frequency representation employed for their definition is the so-called short-time Fourier transform (STFT), whose building blocks are the linear operators of translation and modulation (so-called time-frequency shifts) given by
$$
T_x f(\cdot) = f(\cdot - x) \;\;\; \mbox{ and } \;\;\;
 M_\o f(\cdot) = e^{2\pi i \o \cdot} f(\cdot), \;\;\; x,\,\o \in {\bR}^d.
$$
For $z=(x,\xi)$ we shall also write $\pi(z)f=M_\xi T_x f$.
Indeed, if  $g$ is a non-zero window function in the
Schwartz class $\cS(\rd)$, then the short-time Fourier
transform (STFT) of a a function/tempered distribution $f$
with respect to the window $g$ is given by
\begin{equation}\label{stft} V_gf \phas = \la f,M_\o T_x g\ra =\int_{\Ren}
 f(t)\, {\overline {g(t-x)}} \, e^{-2\pi i\o t}\,dt\, .
\end{equation}
In short: we say that $f$ is in the  modulation space $M^{p,q}_m(\rd)$, $1\leq p,q\leq\infty$, where $m$ is a moderate weight function (hence with at most exponential growth at infinity)  if $V_g f\in L^{p,q}_m(\rdd)$.
So it is clear that the bigger the growth of $m$ the faster  the decay  of $V_g f$ is at infinity.\par
Starting from signal analysis and quantum mechanics, these spaces have then been proved to be the right spaces for
symbols of pseudodifferential operators (see \cite{GH99,labate2,Sjo95,psdoNC09,str06} and references therein),
in particular localization operators (cf. e.g. \cite{locNC09}), Fourier multipliers (\cite{benyi,fabiomultiplier}) and
  recently Fourier integral operators of Schr\"odinger type \cite{mathnacr08,cornic1,cornic2,fio5,fio1}, providing a good framework  for the study
  of PDE's \cite{fio3}.

If we limit the use of time-frequency shifts on a lattice $\Lambda=\a\bZ^d\times\b\bZ^d\subset \rdd$ and fix a window function $g\in L^2(\rd)$, then  the sequence $\G(g,\Lambda)=\{g_{m,n}=M_nT_m g$, $(m,n)\in \Lambda \}$ forms a so-called Gabor system.
The set $\G(g,\Lambda)$   is a Gabor frame, if there exist constants $A,B>0$ such that
\begin{equation*}
A\|f\|_2^2\leq\sum_{(m,n)\in\Lambda}|\langle f,g_{m,n}\rangle|^2\leq B\|f\|^2_2,\qquad \forall f\in L^2(\rd).
\end{equation*}
For details see the next Section $2$, where beside modulation spaces and Gabor frames, we treat classes of symbols and Gelfand-Shilov functions, which represent the natural window functions in our context.\par
The first paper which uses Gabor frames to approximately diagonalize pseudodifferential operators is \cite{rochberg}. Later, motivated by the study of Sj{\"o}strand \cite{Sjo95}, Gr{\"o}chenig proved an almost diagonalization for symbols in the Sj{\"o}strand class, that is the modulation space $M^{\infty,1}$, whose further generalization in \cite{GR} is recalled in Theorem \ref{CR1} below. What is remarkable in  Theorem \ref{CR1} is that the rate of decay in the almost diagonalization, expressed in terms of Wiener amalgam spaces,  characterizes the symbol class, which reveals to be a class of modulation spaces. The discrete representation of a pseudodifferential operator via Gabor matrix is one of the key ingredient for  numerical applications, as shown in \cite{fio3} for the Gabor matrix of Schr\"odinger propagators.\par
Motivated by these characterizations and by numerical applications, in the recent work \cite{CNR12} we study diagonalization of pseudodifferential operators  with symbols of Gevrey, analytic and ultra-analytic type, and  show that Gabor frames surprisingly reveal to be an  efficient tool for representing solutions to hyperbolic and parabolic-type differential equations with constant coefficients. \par
Here we are mainly focused on the theoretical  aspects of the almost diagonalization: as a preliminary step we establish the equivalence between Gevrey, analytic and ultra-analytic regularity of symbols, and their membership to modulation spaces with sub-exponential, exponential and super-exponential weights, respectively. For these aspects, our results are strictly related to those of \cite{ToftGS}. 
The main result of the paper, Theorem \ref{caratterizz}, characterizes the corresponding classes of pseudodifferential operators in terms of exponential decay of the almost diagonalization.
In the last Section $4$, we report on the sparsity result obtained in \cite{CNR12}, which motivates this study for numerical applications.
\vskip0.5truecm
\par\noindent
{\bf Notations.}
The Schwartz class is denoted by
$\sch(\Ren)$, the space of tempered
distributions by  $\sch'(\Ren)$.   We
use the brackets  $\la f,g\ra$ to
denote the extension to $\sch '
(\Ren)\times\sch (\Ren)$ of the inner
product $\la f,g\ra=\int f(t){\overline
{g(t)}}dt$ on $L^2(\Ren)$.%  , the involution
% $g^*$ is $g^*(t) = \overline{g(-t) }$
% and the inverse Fourier transform is
% ${\check  f}(\o)=\Fur^{-1}f (\o)={\hat
% {f}}(-\o)$.\par  We write  $dx\wedge d\xi=\sum_{j=1}^d dx_j\wedge
%   d\xi_j$
%  for the canonical symplectic
%  2-form.\par
%COMM: involution and inverse FT are not used.

Euclidean norm of $ x \in {\bR}^d $ is given by $ |x| = \left( x_1 ^2 + \dots +x_d
^2 \right) ^{1/2}, $ and $ \langle x \rangle = ( 1 + |x|^2 )^{1/2}.$ We write $xy=x\cdot y$ for  the scalar product on
$\Ren$, for $x,y \in\Ren$.

For multiindices $ \alpha, \beta \in {\bN} ^d, $
we have $ |\alpha| = \alpha_1 + \dots + \alpha_d, $
$ \alpha! = \alpha_1! \cdots \alpha_d!, $ $ x^{\alpha} =  x_1 ^{\alpha_1}
\cdots  x_d ^{\alpha_d}, $ $ x \in {\bR}^d, $
and, for $ \beta \leq \alpha, $ i.e. $ \beta_j \leq \alpha_j, $
$ j \in \{ 1,2,\dots, d\}, $ $
{\alpha \choose \beta} $ $ =  {\alpha_1 \choose \beta_1} \cdot \cdots \cdot
{\alpha_d \choose \beta_d}. $
The letter $ C $ denotes a positive constant, not necessarily the same at
every appearance.

The operator of partial differentiation $ \partial $ is given by
$$
\partial^{\alpha} =\partial_x ^{\alpha}=
\partial^{\alpha_1} _{x_1} \cdots  \partial^{\alpha_d} _{x_d}
$$
for all multiindices $ \alpha \in {\bN} ^d $ and all
$ x = (x_1, \dots, x_d) \in {\bR} ^d. $ If
$f$ and $g$ are smooth enough, then the Leibnitz formula holds
$$ \displaystyle
\partial^{\alpha} (f g) (x) = \sum _{\beta \leq \alpha} {\alpha \choose \beta}
\partial^{\alpha - \beta} f (x) \partial^{\beta} g (x).
$$
Observe
\begin{equation}\label{2alfa}
\sum _{\beta \leq \alpha} {\alpha \choose \beta}=2^{|\a|},
\end{equation}
and
\begin{equation}\label{eqn:39}
|\a|!\leq d^{|\a|}\a!.
\end{equation}
 The Fourier
transform is normalized to be ${\hat
  {f}}(\o)=\Fur f(\o)=\int
f(t)e^{-2\pi i t\o}dt$.

For $0< p\leq\infty$ and a  weight $m$, the space $\ell^{p}_m
(\Lambda )$
 is  the (quasi-)Banach
space of sequences $a=\{{a}_{\lambda}\}_{\lambda \in \Lambda }$
on a  lattice $\Lambda$, such that
$$\|a\|_{\ell^{p}_m}:=\left(\sum_{\lambda\in\Lambda}
|a_{\lambda}|^p m(\lambda)^p\right)^{1/p}<\infty
$$
(with obvious changes when $p=\infty$).

We
shall use the notation
$A\lesssim B$ to express the inequality
$A\leq c B$ for a suitable
constant $c>0$, and  $A
\asymp B$  for the equivalence  $c^{-1}B\leq
A\leq c B$.

 \section{Time-frequency analysis and Gelfand-Shilov spaces}
\subsection{Gelfand-Shilov  Spaces}
The Schwartz class $\cS(\rd)$ does not give enough information about how fast a function $f\in \cS(\rd)$ and its derivatives decay at infinity.
This is the main motivation to use subspaces of the Schwartz class, so-called Gelfand-Shilov type spaces, introduced in  \cite{GS} and now available also  in a textbook \cite{NR}. Let us recall their definition and some of their properties.

\begin{definition} Let there be given $ s, r \geq0$.
A function $f\in\cS(\rd)$ is in the Gelfand-Shilov type space $ S^{s} _{r} (\rd) $ if there exist  constants $A,B>0$ such that
\begin{equation}\label{GFdef}
|x^\a\partial^\beta f(x)| \lesssim A^{|\a|}B^{|\beta|}(\a!)^r(\beta!)^s,\quad \a,\beta\in\bN^d.
\end{equation}
\end{definition}
The space $S^{s} _{r}(\rd) $
is nontrivial if and only if $ r + s > 1,$ or $ r + s = 1$ and $r, s > 0$
\cite{GS}.  So the smallest nontrivial space with $r=s$ is provided by $S^{1/2}_{1/2}(\rd)$. Every function of the type $P(x)e^{-a|x|^2}$, with $a>0$ and $P(x)$ polynomial on $\rd$, is in  $S^{1/2}_{1/2}(\rd)$.  Observe that  $S^{s_1} _{r_1}(\rd)\subset S^{s_2} _{r_2}(\rd)$ for $s_1\leq s_2$ and $r_1\leq r_2$.
Moreover, if $f\in S^{s} _{r}(\rd)$, for every $\delta,\gamma\in \bN^d$, $x^\delta\partial^\gamma f \in S^{s}_{r}(\rd)$.
 The action of the Fourier transform on $S^{s} _{r}(\rd) $ interchanges the indices $s$ and $r$, as explained in the following theorem.

\begin{theorem} For $f\in \cS(\rd)$ we have $f\in S^{s} _{r}(\rd) $ if and only if $\hat{f}\in S^{r}_{s}(\rd).$
\end{theorem}

Therefore for $s=r$ the spaces $S^{s} _{s}(\rd)$ are invariant under the action of the Fourier transform.

%A characterization of the class $S^{s} _{r}(\rd)$ is detailed below.
%
\begin{theorem} \label{simetria} Assume $ s>0, r>0, s+r\geq1$. For $f\in  \cS(\rd)$, the following conditions are equivalent:
\begin{itemize}
\item[a)] $f \in  S^{s} _{r}(\rd)$ .
\item[b)]  There exist  constants $A, B>0,$ such that
$$ \| x^{\a}  f \|_{L^\infty} \lesssim A^{|\a|} (\a !) ^{r}  \quad\mbox{and}\quad  \| \o^{\b}  \hat{f} \|_{L^\infty} \lesssim  B^{|\beta|} ( \b!) ^{s},\quad \a,\b\in \bN^d. $$
\item[c)]  There exist  constants $A, B>0,$ such that
$$ \| x^{\a}  f \|_{L^\infty} \lesssim A^{|\a|} (\a !) ^{r}  \quad\mbox{and}\quad  \| \partial^{\b}  f \|_{L^\infty} \lesssim  B^{|\beta|} ( \b!) ^{s},\quad \a,\b\in \bN^d. $$

\item[d)] There exist  constants $h, k>0,$ such that
$$
 \|f  e^{h  |x|^{1/r}}\|_{L^\infty} < \infty \quad\mbox{and}\quad \| \hat f  e^{k  |\o|^{1/s}}\|_{L^\infty} < \infty.$$
\end{itemize}
\end{theorem}

A suitable window class for weighted modulation spaces (see Definition \ref{prva} below) is the Gelfand-Shilov  space $\Sigma^1_1(\rd)$, consisting of  functions $f\in\cS(\rd)$ such that for {\it every}  constant $A>0$ and $B>0$
\begin{equation}\label{sigmadef}
|x^\a\partial^\beta f(x)| \lesssim A^{|\a|}B^{|\beta|}\a!\beta!,\quad \a,\beta\in\bN^d.
\end{equation}
We have $S^s_s(\rd)\subset\Sigma^1_1(\rd)\subset S^1_1(\rd)$ for every $s<1$.
Observe that the characterization of Theorem \ref{simetria} can be adapted to $\Sigma^1_1(\rd)$ by replacing the words ``there exist'' by ``for every'' and taking $r=s=1$. Similarly, one can define $\Sigma^r_s(\rd)$ for any $r>0, s>0$, with $r+s>1$. \par
Let us underline the following property, proved in \cite[Proposition 6.1.5]{NR} and \cite[Proposition 2.4]{CNR12},
 exhibiting two  equivalent ways of expressing the exponential decay of a continuous function $f$ on $\rd$.
\begin{proposition}\label{equi} Consider $r>0$ and let $h$ be a continuous function on $\rd$. Then the following conditions are equivalent:\\
(i) There exists a constant $\eps>0$ such that
\begin{equation}\label{eqn:A.7}
|h(x)|\lesssim e^{-\eps |x|^{\frac1r}},\quad x\in\rd,
\end{equation}
\noindent
(ii)  There exists a constant $C>0$ such that
\begin{equation}\label{powerdecay}
|x^\a h(x)|\lesssim C^{|\a|}(\a!)^{r},\quad x\in\rd, \,\a\in\bN^d.
\end{equation}
\end{proposition}

\begin{remark}\label{linkconst}
To be precise, the relation  between the constants  $\eps$ and $C$ is as follows.
Assuming \eqref{eqn:A.7}, then \eqref{powerdecay} is satisfied with $C=\displaystyle{\left(\frac { r d}{\eps}\right)^r}$. Viceversa, \eqref{powerdecay}
implies \eqref{eqn:A.7} for any $\epsilon<r (d C)^{-\frac{1}{r}}$. Also, it follows from the proof that the constant implicit in the notation $\lesssim$ in \eqref{eqn:A.7} depends only on the corresponding one in \eqref{powerdecay} and viceversa.  \par
\end{remark}

The strong dual spaces of $S^s_r(\rd)$ and  $\Sigma^1_1(\rd)$  are the so-called spaces of tempered ultra-distributions,  denoted by $(S^s_r)'(\rd)$ and $(\Sigma^1_1)'(\rd)$, respectively. Notice that they contain the space of tempered distributions $\cS'(\rd)$.
Moreover the spaces $S^s_r(\rd)$ are nuclear spaces \cite{mitjagin}. This provides a kernel theorem for Gelfand-Shilov spaces, cf. \cite{treves}.
\begin{theorem}\label{kernelT} There exists an isomorphism between the space of linear continuous maps $T$ from $S^s_r(\rd)$ to $(S^s_r)'(\rd)$ and $(S^s_r)'(\rdd)$, which associates to every $T$ a kernel $K_T\in (S^s_r)'(\rdd)$ such that
$$\la Tu,v\ra=\la K_T, v\otimes \bar{u}\ra,\quad \forall u,v \in S^s_r(\rd).$$
$K_T$ is called the kernel of $T$.
\end{theorem}

\subsection{Gabor frames and time-frequency representations.}  We recall the basic
concepts  of \tfa\ and  refer the  reader to \cite{grochenig} for the full
details. %  a complete
% introduction of time-frequency analysis.
Consider a distribution $f\in\cS '(\rd)$
and a Schwartz function $g\in\cS(\rd)\setminus\{0\}$ (the so-called
{\it window}).
The short-time Fourier transform (STFT) of $f$ with respect to $g$ is
defined in \eqref{stft}.
 The  \stft\ is well-defined whenever  the bracket $\langle \cdot , \cdot \rangle$ makes sense for
dual pairs of function or (ultra-)distribution spaces, in particular for $f\in
\cS ' (\rd )$ and $g\in \cS (\rd )$, $f,g\in\lrd$, $f\in
(\Sigma_1^1) ' (\rd )$ and $g\in \Sigma_1^1 (\rd )$ or $f\in
(S^s_r) ' (\rd )$ and $g\in S^s_r (\rd )$.
%Let us recall the covariance formula for the \stft\  that will be used in
%the sequel
%  \begin{equation}
%    \label{eql2}
%   V_{ {g}}(M_\eta T_y{f})(x,\xi)  = e^{-2\pi
%     i(\xi-\eta)y}(V_gf)(x-y,\xi-\eta), \qquad x, y,\o , \eta \in \Ren.
%  \end{equation}

Another time-frequency representation we shall use is the (cross-)Wigner distribution of $f,g\in \lrd$, defined as
\begin{equation}\label{cross-wigner}
W(f,g)\phas=\intrd f\Big(x+\frac t2\Big){\overline{g\Big(x-\frac t2\Big)}} e^{-2\pi i t\o}\,dt.
\end{equation}
If we set $\breve{g}(t)=g(-t)$, then the relation between cross-Wigner distribution and short-time Fourier transform is provided by
\begin{equation}\label{wignerSTFT}
W(f,g)\phas=2^d e^{4\pi i x\o}V_{\breve{g}}f(2x,2\o).
\end{equation}

\par

For the discrete description of function spaces and operators we use
Gabor frames. Let $\Lambda=A\zdd$  with $A\in GL(2d,\R)$ be a lattice
of the time-frequency plane.
 The set  of
time-frequency shifts $\G(g,\Lambda)=\{\pi(\lambda)g:\
\lambda\in\Lambda\}$ for a  non-zero $g\in L^2(\rd)$ is called a
Gabor system. The set $\G(g,\Lambda)$   is
a Gabor frame, if there exist
constants $A,B>0$ such that
\begin{equation}\label{gaborframe}
A\|f\|_2^2\leq\sum_{\lambda\in\Lambda}|\langle f,\pi(\lambda)g\rangle|^2\leq B\|f\|^2_2,\qquad \forall f\in L^2(\rd).
\end{equation}
 If \eqref{gaborframe} is satisfied, then there exists a dual window $\gamma\in L^2(\rd)$, such that $\cG(\gamma,\Lambda)$ is a frame, and every $f\in L^2(\rd)$ has the frame expansions
 \[
 f=\sum_{\lambda\in\Lambda}\langle f,\pi(\lambda)g\rangle\pi(\lambda)\gamma=\sum_{\lambda\in\Lambda}\langle f,\pi(\lambda)\gamma\rangle \pi(\lambda)g
 \]
 with unconditional convergence in $L^2(\rd)$.

 Eventually, we list some results about time-frequency analysis of Gelfand-Shilov functions, cf. \cite{medit,GZ,T2}:
 \begin{equation}\label{zimmermann1}
 f,g\in S^s_s(\rd),\ s\geq1/2 \Rightarrow V_g f\in S^s_s(\rdd),
 \end{equation}
 If $g\in S^s_s(\rd$), $s\geq1/2$, then
 \begin{equation}\label{zimmermann2}
   f\in S^s_s(\rd)\Leftrightarrow |V_g(f)(z)|\lesssim e^{-\epsilon |z|^{1/s}}\ \mbox{for some} \,\,\epsilon>0.
 \end{equation}

   Since we have not found a precise reference in the literature, we also recall and present the proof of the following properties concerning the Wigner distribution.
\begin{proposition} We have:
\begin{align}
\qquad f,g\in S^s_s(\rd),\ s\geq1/2&\Rightarrow W(f,g)\in S^s_s(\rdd),\label{wigner-gelfand}\\
\qquad f,g\in \Sigma_1^1(\rd) &\Rightarrow W(f,g)\in \Sigma^1_1(\rdd).\label{wigner-gelfand2}
 \end{align}
 \end{proposition}
\begin{proof}   We prove  \eqref{wigner-gelfand}, formula \eqref{wigner-gelfand2} follows by similar techniques.
The assumption $g\in S^s_s(\rd)$ trivially implies $\breve{g}\in S^s_s(\rd)$ and \eqref{zimmermann1} proves that the short-time Fourier transform $V_{\breve{g}}g$ is in $S^s_s(\rdd)$. To prove that $\Phi=W(g,g)\in S^s_s(\rdd)$, we use the equivalence $a)\Leftrightarrow c)$ in Theorem \ref{simetria} and  the connection  between $V_{\breve{g}}g$ and  the Wigner distribution in \eqref{wignerSTFT}. First, let us prove the first inequality in Theorem \ref{simetria} c). For $z=\phas\in\rdd$, we have
\begin{align*} |z^\a W(g,g)(z)|&=2^d |z^\a V_{\breve{g}}g(2 z)|=2^{d-|\a|}|(2 z )^\a V_{\breve{g}}g(2 z)|\\
&\leq 2^{d-|\a|}\|z^\a V_{\breve{g}}g\|_{L^\infty}\lesssim 2^{d-|\a|}A^{|\a|}(\a!)^s\lesssim \left(\frac A 2\right)^{|\a|}(\a!)^s.
\end{align*}
The second inequality in Theorem \ref{simetria} c) requires more computations but the techniques are the same as before.
Using Leibniz formula  for the $x$-derivatives and the $\o$-derivatives separately,
\begin{align*} \partial_\o^\a \partial_x^\b W(g,g)\phas&= \partial_\o^\a (\partial_x^\b (2^d e^{4\pi i x\o}V_{\breve{g}}g(2x,2\o))\\&=2^{d}\partial_\o^\a\left(\sum _{\gamma \leq \b} {\b \choose \gamma}
(4\pi i \o)^{\b - \gamma}  e^{4\pi i x\o} 2^{|\gamma|}\partial_x^\gamma(V_{\breve{g}}g)(2x,2\o)\right)\\
&= 2^{d}\sum _{\gamma \leq \b} {\b \choose \gamma}\sum _{\delta_1+\delta_2+\delta_3= \a\atop \delta_1\leq \beta-\gamma} \frac{\a!}{\delta_1!\delta_2!\delta_3!}
\partial^{\delta_1}_\o((4\pi i \o)^{\b - \gamma})   (4\pi i x)^{\delta_2}e^{4\pi i x\o}\\
&\qquad\qquad\qquad  \times\,\,2^{|\gamma|+|\delta_3|} \partial_\o^{\delta_3}\partial_x^\gamma(V_{\breve{g}}g)(2x,2\o).
\\
&= 2^{d}\sum _{\gamma \leq \b} {\b \choose \gamma}\sum _{\delta_1+\delta_2+\delta_3= \a\atop\delta_1\leq\beta-\gamma} 2^{|\gamma|+|\delta_3|} (4\pi i)^{|\beta-\gamma|+|\delta_2|}\frac{\a!}{\delta_1!\delta_2!\delta_3!} \frac{(\beta-\gamma)!}{(\beta-\gamma-\delta_1)!}\\ &\,\qquad\qquad\qquad
\times\o^{\b - \gamma-\delta_1}   x^{\delta_2}e^{4\pi i x\o}
 \partial_\o^{\delta_3}\partial_x^\gamma(V_{\breve{g}}g)(2x,2\o).
\end{align*}
Now, we have
\begin{align*}|\o^{\b-\gamma-\delta_1} x^{\delta_2} \partial_\o^{\delta_3}\partial_x^\gamma(V_{\breve{g}}g)(2x,2\o)|&\lesssim A^{|\alpha|+|\beta|}((\beta-\gamma-\delta_1+\delta_2+\delta_3+\gamma)!)^s\\&
= A^{|\alpha|+|\beta|}((\beta+\alpha-2\delta_1)!)^s,
\end{align*}
where we used $\delta_2+\delta_3=\a-\delta_1$. Since from \eqref{2alfa}
\[
\frac{(\beta-\gamma)!}{(\beta-\gamma-\delta_1)!}\leq 2^{|\beta-\gamma|}\delta_1!\leq 2^{|\beta-\gamma|}((2\delta_1)!)^{1/2}\leq 2^{|\beta-\gamma|}((2\delta_1)!)^{s},
\]
together with $\sum _{\gamma \leq \b} {\b \choose \gamma}=2^{|\b|}$ and $\sum _{\delta_1+\delta_2+\delta_3= \a} \frac{\a!}{\delta_1!\delta_2!\delta_3!}= 3^{|\a|}$,
we obtain the desired estimate
$$|\partial_\o^\a \partial_x^\b W(g,g)\phas|\lesssim C^{|\a+\b|} ((\a+\b)!)^{s}.$$
\end{proof}

\subsection{Modulation Spaces}\label{moddef}
Modulation spaces measure the decay of the STFT on the time-frequency (phase space) plane and were introduced by Feichtinger in the 80's \cite{F1}, for weight of sub-exponential growth at infinity. The study of weights of exponential growth at infinity was developed in \cite{medit,T2}.\par
\emph{Weight Functions.}  In the sequel $v$ will always be a
continuous, positive,  even, submultiplicative   function
(submultiplicative weight), i.e., $v(0)=1$, $v(z) =
v(-z)$, and $ v(z_1+z_2)\leq v(z_1)v(z_2)$, for all $z,
z_1,z_2\in\Renn.$
Submultiplicativity implies that $v(z)$ is \emph{dominated} by an exponential function, i.e.
\begin{equation} \label{weight}
 \exists\, C, k>0 \quad \mbox{such\, that}\quad  1\leq v(z) \leq C e^{k |z|},\quad z\in \rdd.
\end{equation}

For instance,  weights of the form
\begin{equation} \label{BDweight}
v(z) =   e^{s|z|^b} (1+|z|)^a \log ^r(e+|z|)
\end{equation}
 where $a,r,s\geq 0$, $0\leq b \leq 1$, satisfy the
above conditions.\par

We denote by $\mathcal{M}_v(\rdd)$ the space of $v$-moderate weights on $\rdd$;
these  are positive and measurable functions $m$ satisfying $m(z+\zeta)\leq C
v(z)m(\zeta)$ for every $z,\zeta\in\rdd$.

\begin{definition}  \label{prva}
Given  $g\in\Sigma^1_1(\rd)$, a  weight
function $m\in\cM _v(\rdd)$, and $1\leq p,q\leq
\infty$, the {\it
  modulation space} $M^{p,q}_m(\Ren)$ consists of all tempered
ultra-distributions $f\in(\Sigma^1_1)' (\rd) $ such that $V_gf\in L^{p,q}_m(\Renn )$
(weighted mixed-norm spaces). The norm on $M^{p,q}_m(\rd)$ is
\begin{equation}\label{defmod}
\|f\|_{M^{p,q}_m}=\|V_gf\|_{L^{p,q}_m}=\left(\int_{\Ren}
  \left(\int_{\Ren}|V_gf(x,\o)|^pm(x,\o)^p\,
    dx\right)^{q/p}d\o\right)^{1/q}  \,
\end{equation}
(obvious changes if $p=\infty$ or $q=\infty$).
\end{definition}
For $f,g\in \Sigma^1_1(\rd)$ the above integral is convergent and thus $\Sigma^1_1(\rd)\subset M^{p,q}_m(\rd) $, $1\leq p,q \leq\infty$, cf.\ \cite{medit}, with dense inclusion when $p,q<\infty$, cf.\ \cite{elena07}. When $p=q$, we simply write $M^{p}_m(\rd)$ instead of $M^{p,p}_m(\rd)$. The spaces $M^{p,q}_m(\rd)$ are Banach spaces and every nonzero $g\in M^{1}_v(\rd)$ yields an equivalent norm in \eqref{defmod} and so $M^{p,q}_m(\Ren)$ is independent on the choice of $g\in  M^{1}_v(\rd)$.

We observe that  these  properties of modulation spaces do depend on the fact that the weight functions involved have at most exponential growth at infinity. Indeed, as well known, if we consider super exponential weights at infinity, say, e.g. $m(z)=e^{k |z|^b}$, with $k>0$ and $1<b\leq 2$, then the related modulation spaces $M^{p,q}_m(\rd)$ can still be defined by taking Gelfand-Shilov windows, but their definition depends on the choice of the window $g$, because of the loss of sub-multiplicativity of the weight.
%\par We  recall the inversion formula for
%the STFT (see  (\cite[Proposition 11.3.2]{grochenig} and \cite[Proposition 2.6]{medit} for exponential weights): assume $g\in M^{1}_v(\rd)\setminus\{0\}$,
% $f\in M^{p,q}_m(\rd)$, then
%\begin{equation}\label{invformula}
%f=\frac1{\|g\|_2^2}\int_{\R^{2d}} V_g f(x,\o)M_\o T_x g\, dx\,d\o,
%\end{equation}
%and the  equality holds in $M^{p,q}_m(\rd)$ (observe that $M^{1}_v(\rd)\subset M^2(\rd)=L^2(\rd)$, so $\|g\|_2<\infty$).
\subsection{Ultra-Modulation Spaces}
In this section we present a new definition of modulation spaces, which considers also weights of super-exponential growth at infinity (of a very particular form).
For related constructions, we refer to \cite{elena07,GZ} and more recently to \cite{ToftGS}.
Let us first introduce  the weights:
\begin{equation}\label{pesiw}
w_{s,\eps}(z):=e^{\eps|z|^{\frac1s}},\quad z \in\rd,\,\, s\geq 1/2,\,\, \eps>0,
\end{equation}
and the weight class $\cN(\rdd)=\{
m_{s,\eps}(x,\o)=(1\otimes
w_{s,\eps})(x,\o)=w_{s,\eps}(\o),\,x,\xi\in\rd,\,\eps>0,s\geq
1/2\}$, which we may identify with the set $(0,+\infty)\times [1/2,+\infty)$ of the corresponding $\eps,s$ parameters.
Thus, a weight $m_{s,\eps}\in\cN(\rdd)$ grows faster than
exponentially at infinity in the frequency variable
whenever $1/2\leq s<1$. Notice that the bound $s=1/2$ is
admitted. We use the class $\cN(\rdd)$ as weight class for
modulation spaces. We limit our study to weights in the
frequency variable for simplicity. On the other hand, observe that a limit in
enlarging the definition to the space variables  is imposed by
Hardy's theorem:
 if $m(z)\geq C e^{c |z|^2}$, for $z=(x,\o)\in\rdd$ and some $c>\pi/2$, then the corresponding \modsp s  are trivial \cite{GZ}.

\begin{definition}\label{defmodnorm}
Let $\eps>0$, $s\geq1/2$ and  $m_{s,\eps}\in \cN(\rdd)$. Consider  a non-zero \emph{window} function $g$ in $\cS^{s}_{s}(\rd)$. For  $1\leq p,q\leq
\infty$,  we define
$M^{p,q}_{m_{s,\eps},g}(\rd)$  the subspace of $f\in (\Sigma^{1}_{1})'(\rd)$ such that the integrals in \eqref{defmod} are finite (with obvious changes if either $p=\infty$ or $q=\infty$).
\end{definition}

\noindent
Notice that:\\
(i) If $s>1/2$, $f,g \in\cS ^{1/2}_{1/2}(\rd )$, the  integral in \eqref{defmod} is
convergent thanks to \eqref{zimmermann1} and Theorem \ref{simetria}, item $d)$. Indeed, we can find
 $h>0$   such that $\|V_gf  e^{h |\cdot|^{2}}\|_{L^\infty} < \infty$ and
\begin{eqnarray*}
&&\int_{\Ren}
  \left ( \int_{\Ren}|V_gf\phas|^p m_{s,\eps}(x,\o)^p\,
    dx\right)^{q/p}d\o\\
   &&\quad\leq C\,\|(V_gf) e^{h |\cdot|^2}\|_{L^\infty}
   \int_{\Ren}
  \left( \int_{\Ren}|m_{s,\eps}(x,\o)|^p e^{-h p |\phas|^2}\,
    dx\right) ^{q/p}\!\!\!  d\o < \infty.
\end{eqnarray*}
\noindent (ii) For $m_{s,\eps}\in \cM_v(\rdd)$, hence
$s\geq 1$, then $M^{p,q}_{m_{s,\eps},g}(\rd)$ is the
subspace of ultra-distribution
$(\Sigma^{1}_{1})'(\rd)$ defined in Definition \ref{prva}. So we come back to the classical modulation spaces. This justify the same notation.\\
\noindent
(iii) The definition of $M^{p,q}_{m_{s,\eps},g}(\rd)$ may depend on the choice of the window
function $g$. However,  if
$s\geq1$, the definition of $M^{p,q}_{m_{s,\eps},g}(\rd)$ does not depend on
$g$:  the class of admissible windows
can be enlarged to $M^1_{m_{s,\eps}}(\rd)$ \cite[Proposition 1]{elena07}, so we recapture the standard definition \ref{prva}.

\par
\subsection{Time-frequency analysis of Gevrey-analytic and ultra-analytic symbols}\label{tfc} Here we present the results obtained in \cite[Section 3]{CNR12}, where  the smoothness and the growth of a function $f$ on $\rd$ is characterized in terms of the decay of its STFT $V_g f$, for a suitable window $g$.
\begin{theorem}\label{teo1}
Consider $s>0$, $m\in \mathcal{M}_v(\rd)$, $g\in M^1_{v\otimes 1}(\rd)\setminus\{0\}$ such that there exists $C_g>0$,
\begin{equation}\label{finestra}
\|\partial^\a g\|_{L^1_v(\rd)}\lesssim C_g^{|\a|}(\a!)^s,\quad\a\in\bN^d.
\end{equation}
For $f\in\cC^\infty(\rd)$ the following conditions are equivalent:\\
(i) There exists a constant $C_f>0$ such that
\begin{equation}\label{smoothf}
|\partial^\a f(x)|\lesssim m(x)C_f^{|\a|}(\a!)^s,\quad x\in\rd,\,\a\in\bN^d.
\end{equation}
\noindent
(ii)  There exists a constant $C_{f,g}>0$ such that
\begin{equation}\label{STFTf}
|\o^\a V_gf\phas|\lesssim m(x)C_{f,g}^{|\a|}(\a!)^s,\quad \phas\in\rdd,\,\a\in\bN^d.
\end{equation}
(iii)  There exists a constant $\eps>0$ such that
\begin{equation}\label{STFTeps}
|V_gf\phas|\lesssim m(x)e^{-\eps|\o|^{\frac1s}},\quad \phas\in\rdd,\,\a\in\bN^d.
\end{equation}
\end{theorem}
We say that the function $f$ is Gevrey if $s>1$, analytic if $s=1$, and ultra-analytic when $s<1$.
\begin{remark} We observe that the assumption $m\in \mathcal{M}_v(\rd)$ in the previous theorem is essential to obtain the equivalence: indeed in the proof we exploit the $v$-moderateness of $m$.\par
\end{remark}
For simplicity, from now on we shall assume $m=v=1$.\par
A natural question is whether we may find window functions satisfying \eqref{finestra}. To find the answer, we recall the following characterization of Gelfand-Shilov spaces.
\begin{proposition}\label{pro3.2} Let $g\in\cS(\rd)$. We have $g\in S^s_r(\rd)$, with $s,r>0$, $r+s\geq1$, if and only if there exist constants $A>0$, $\epsilon>0$ such that
$$|\partial^\a g(x)|\lesssim A^{|\a|} (\a!)^s e^{-\eps|x|^{\frac1r}},\quad x\in\rd,\,\,\a\in\bN^d.
$$
We have $g\in \Sigma^1_1(\rd)$ if and only if, for every $A>0$, $\epsilon>0$,
$$|\partial^\a g(x)|\lesssim A^{|\a|} \a! e^{-\eps|x|},\quad \, x\in\rd,\,\,\a\in\bN^d.
$$
\end{proposition}

Hence every $g\in S^s_r(\rd)$ with $s>0$, $0<r<1$, $s+r\geq 1$, satisfies \eqref{finestra} for every submultiplicative weight $v$ (see \eqref{weight}). The same holds true if $g\in\Sigma_1^1(\rd)$ and $s\geq1$.

\section{Almost Diagonalization for Pseudodifferential operators}\label{section4}
In this section we first present the almost diagonalization for pseudodifferential operators having (ultra-)analytic symbols obtained in \cite[Section 4]{CNR12}. This result can be seen as an extension of the almost diagonalization for pseudodifferential operators obtained in \cite{GR}, where  only the Gevrey-analytic case  was  discussed.\par
Secondly, we exhibit a new characterization involving ultra-modulation spaces.\par
Recall the Weyl form $\sigma^w$ of a  pseudodifferential operator (so-called Weyl operator or Weyl transform) with symbol $\sigma\phas$ on $\rdd$,  formally defined by
\begin{equation}\label{Weyl}
\sigma^w f (x)=\intrd \sigma\left(\frac{x+y}2,\eta\right)e^{2\pi i (x-y)\eta}f(y)\,dy d\eta.
\end{equation}
Using the Kernel Theorem  for Gelfand-Shilov spaces (Theorem \ref{kernelT}), we have a characterization of linear continuous operators $T: S^s_r(\rd) \to (S^s_r)'(\rd)$. In particular, any such operator $T$ can be represented as a pseudodifferential  operator in the Weyl form, with $\sigma\in(S^r_s)'(\rdd)$. Thus we shall exhibit our results for Weyl operators.

It was proved in \cite{charly06,GR} that Gabor frames allow to discretize  a continuous operator from
 $\cS(\rd)$ to $\cS'(\rd)$ into  an infinite matrix that captures the  properties of the original operator.
 In particular, the authors consider symbols in the modulation spaces $M^{\infty,q}_{1\otimes v\circ j^{-1}}$, where $v$ is a continuous
and submultiplicative weight function on $\rdd$ and  $j$ is  the rotation on $\rdd$:
$$j(z_1,z_2)=(z_2,-z_1),\quad(z_1,z_2)\in\rdd.$$
Fix a Banach algebra of sequences $\ell^q_v(\Lambda)$, where  $\Lambda$ is a lattice on $\rdd$ with relatively compact fundamental domain $\mathcal{Q}$  containing the origin. Then a function $H\in L^\infty_{loc}(\rdd)$ belongs to the Wiener amalgam space $W(\ell^q_v)(\rdd)$ if the sequence $h(\lambda):=\mbox{ess\,sup}_{u\in\lambda+\mathcal{Q}} H(u)$ is in $\ell^q_v(\Lambda)$.
The almost diagonalization for pseudodifferential operators in  \cite{charly06,GR} can be extended easily to modulation spaces with exponential weights $v(z)=e^{c|z|}$ \cite{medit}. Indeed one can use the same pattern of \cite[Theorem 4.1]{GR}, provided that the symbol  space of distributions $\cS'(\rdd)$ is replaced by the space of ultra-distributions $(\Sigma_1^1)'(\rdd)$. Thus  we rephrase \cite[Theorem 4.1]{GR} in our context as follows.

\begin{theorem}[\cite{charly06,GR}]\label{CR1}  Assume  that $\G(g,\Lambda)$ is a
 frame for $L^2(\rd)$ with  $g\in M^1_v(\rd)$.  Then the
 following statements  are equivalent for $\sigma\in(\Sigma^1_1)'(\rdd)$:
\par {\rm
(i)} $\sigma\in M^{\infty,q}_{1\otimes v\circ j^{-1}}(\rdd)$.\par
{\rm (ii)} There exists a function $H\in W(\ell^q_v)(\rdd)$  such that
\begin{equation}\label{CC1} |\langle \sigma^w \pi(z)
g,\pi(w)g\rangle|\lesssim H(w-z),\qquad \forall
w,z\in\rdd.
\end{equation}
{\rm (iii)} There exists a sequence $h\in \ell^q_v(\Lambda)$ such that
\begin{equation}\label{CC2} |\langle \sigma^w \pi(\lambda)
g,\pi(\mu)g\rangle|\lesssim h(\mu-\lambda),\qquad \forall
\lambda,\mu\in \Lambda.
\end{equation}
\end{theorem}
The  Gabor kernel of $\sigma^w$: $\langle \sigma^w \pi(z)
g,\pi(w)g\rangle$ is referred as \emph{continuous} Gabor matrix, whereas  $\langle \sigma^w \pi(\lambda)
g,\pi(\mu)g\rangle$ is the Gabor matrix of $\sigma^w$.\par
Our  concern is in the case of a continuous Gabor matrix dominated by a function
$H\in W(\ell^\infty_v)(\rdd)$.
In this special case, Theorem \ref{CR1} can be simplified as follows.

\begin{theorem}\label{E7}  Assume  that $\G(g,\Lambda)$ is a
 frame for $L^2(\rd)$ with  $g\in M^1_v(\rd)$.  Then the
 following statements  are equivalent for $\sigma\in(\Sigma^1_1)'(\rdd)$:
\par {
(i)} $\sigma\in M^{\infty}_{1\otimes v\circ j^{-1}}(\rdd)$.\par
{(ii)}The  continuous Gabor matrix satisfies
\begin{equation}\label{CC1bis} |\langle \sigma^w \pi(z)
g,\pi(w)g\rangle|\lesssim v(w-z),\qquad \forall
w,z\in\rdd.
\end{equation}
{ (iii)}  The   Gabor matrix  satisfies
\begin{equation}\label{CC2bis} |\langle \sigma^w \pi(\lambda)
g,\pi(\mu)g\rangle|\lesssim v(\mu-\lambda),\qquad \forall
\lambda,\mu\in \Lambda.
\end{equation}
\end{theorem}

Our main aim here is to give a counterpart of Theorem \ref{E7} for weights of super-exponential growth.
First, let us survey the new results of  \cite[Section 4]{CNR12} which provide the arguments for the characterization of Theorem \ref{caratterizz} below.\par

The crucial relation between the action of the Weyl operator $\sigma^w$ on time-frequency shifts and the short-time Fourier transform of its symbol, contained in \cite[Lemma 3.1]{charly06} can now be extended to Gelfand-Shilov spaces  and their dual spaces as follows.
\begin{lemma} Consider $s\geq1/2$, $g\in S^s_s(\rd)$, $\Phi=W(g,g)$. Then, for $\sigma\in (S^s_s)'(\rdd)$,
\begin{equation}\label{311}
|\la\sigma^w \pi(z)g,\pi(w) g\ra|=\left|V_\Phi \sigma\left(\frac{z+w}2,j(w-z)\right)\right|=|V_\Phi\sigma(u,v)|
\end{equation}
and
\begin{equation}\label{312}
|V_\Phi \sigma(u,v)|=\left|\la\sigma^w \pi\left(u-\frac12 j^{-1}(v)\right)g,\pi\left(u+\frac12 j^{-1}(v)\right) g\ra\right|.
\end{equation}
\end{lemma}
\begin{proof}
Since $\Phi=W(g,g)\in S^s_s(\rdd)$ for $g\in S^s_s(\rd)$ by \eqref{wigner-gelfand}, the duality $\la \sigma,\pi(u,v)\Phi\ra_{(S^s_s)'\times S^s_s}$ is well-defined so that the short-time Fourier transform $V_\Phi \sigma(u,v)$ makes sense.  The rest of the proof is analogous to \cite[Lemma 3.1]{charly06}.
\end{proof}

Given a pseudodifferential operator $\sigma^w$, with smooth
symbol $\sigma\in\cC^\infty(\rdd)$, we  exhibit the
connection between the Gevrey, analytic or ultra-analytic
regularity of $\sigma$ (as considered in Theorem \ref{teo1}, the dimension being now $2d$), the (ultra-)modulation space $\sigma$
belongs to (cf. Definition \ref{defmodnorm}), and the decay of the continuous Gabor matrix of
$\sigma^w$.\par
 Recall the weights $w_{s,\eps}$ defined in \eqref{pesiw} and here used as functions over $\rdd$.
\begin{theorem}\label{caratterizz} Let $s\geq1/2$ and consider a window function $g\in S^s_s(\rd)$. Set $\Phi=W(g,g)$. Then the following properties are equivalent for $\sigma\in\cC^\infty(\rdd)$:
\par {(i)}  There exists $\eps>0$ such that $\sigma\in M^{\infty}_{1\otimes w_{s,\eps},\Phi}(\rdd)$.\par
{(ii)} There exists $C>0$ such that the symbol $\sigma$ satisfies
\begin{equation}\label{simbsmooth} |\partial^\a \sigma(z)|\lesssim  C^{|\a|}(\a!)^{s}, \quad \forall\, z\in\rdd,\,\forall \a\in\bN^{2d}.\end{equation}
\par{(iii)} There exists $\eps>0$ such that
\begin{equation}\label{unobis2s}|\langle \sigma^w \pi(z)
g,\pi(w)g\rangle|\lesssim w_{s,-\eps}(w-z),\qquad \forall\,
z,w\in\rdd.
\end{equation}
\end{theorem}
\begin{proof} $(i)\Rightarrow (iii)$.  Assuming $\sigma\in M^{\infty}_{1\otimes w_{s,\eps},\Phi}(\rdd)$, that is
$$\sup_{u,v\in\rdd} |V_\Phi \sigma(u,v)| e^{\eps|v|^{\frac1s}}<\infty
$$
 and using  \eqref{312}, we obtain the claim:
 \begin{align*}|\langle \sigma^w \pi(z)
g,\pi(w)g\rangle|&= \left|V_\Phi \sigma (\frac{w+z}2, j(w-z))\right|\\
&\leq \sup_{u\in\rdd} |V_\Phi \sigma (u, j(w-z))|
\lesssim  e^{-\eps|w-z|^{\frac1s}}.
\end{align*}
$(iii)\Rightarrow (i)$. Relation \eqref{312} and the decay assumption \eqref{unobis2s} give
\begin{align}\label{e88}
|V_\Phi \sigma(u,v)|&=\left|\la\sigma^w \pi\left(u-\frac12 j^{-1}(v)\right)g,\pi\left(u+\frac12 j^{-1}(v)\right)g\ra\right|\\
&\lesssim  e^{-\eps | j^{-1}(v)|^{\frac1s}}= e^{-\eps | v|^{\frac1s}}.
\end{align}
The equivalences $(ii)\Leftrightarrow (iii)$ are proved in Theorem \cite[Theorem 4.2]{CNR12}.
For sake of clarity, we give a brief sketch of this result. The techniques are similar to \cite[Theorem 3.2]{charly06}. \par
$(ii)\Rightarrow (iii)$. The window $\Phi=W(g,g)\in  S^s_s(\rdd)$, for $g\in S^s_s(\rd)$ by \eqref{wigner-gelfand}, and satisfies the assumptions of Theorem \ref{teo1}. Hence, using the equivalence \eqref{smoothf} $\Leftrightarrow$ \eqref{STFTeps}, the  assumption \eqref{simbsmooth} is equivalent to the following decay estimate of the corresponding short-time Fourier transform
$$|V_\Phi \sigma(u,v)|\lesssim  e^{-\eps |v|^{\frac1s}}, \quad u,v\in\rdd,
$$
for a suitable $\eps>0$, hence $$|V_\Phi \sigma\big(\frac{z+w}2,j(w-z)\big)|\lesssim e^{-\eps |j(w-z)|^{\frac1s}}= e^{-\eps |w-z|^{\frac1s}}
$$
which combined with   \eqref{311} yields
 $(iii)$.\par
$(iii)\Rightarrow (ii)$. We use relation \eqref{312} and the decay assumption \eqref{unobis2s} again, which give \eqref{e88}
and using the equivalence  \eqref{smoothf} $\Leftrightarrow$ \eqref{STFTeps} we obtain the claim.
\end{proof}
\par
\begin{remark} a) We observe that the constant $\eps>0$, which depends on the choice of $g\in S^s_s(\rd)$, is the same in $(i)$ and $(iii)$. Whereas the link between $\eps>0$ and the constant $C>0$ in $(ii)$ is specified in Remark \ref{linkconst}.\\
 b) If we consider $s\geq1$, that is the symbol $\sigma$ is a Gevrey or an analytic symbol, then the  equivalence $(i)\Leftrightarrow(ii)$ follows by combining Theorems \ref{E7} and \cite[Theorem 4.2]{CNR12}.
\end{remark}

\par
Of course the estimate \eqref{unobis2s} implies  the discrete analog \eqref{unobis2discr} below. The vice versa is not obvious and requires the existence of a  Gabor frame $\G(g,\Lambda)$ having $g\in  S^{1/2} _{1/2} (\rd)$ and a dual window $\gamma \in  S^{1/2} _{1/2} (\rd)$ as well.  We call such a Gabor frame a \emph{Gabor super-frame}. The existence of Gabor super-frames is due to  a  result obtained by   Gr{\"o}chenig and  Lyubarskii in \cite{GL09}.  They find sufficient conditions on the lattice $\Lambda=A \bZ^2$, $A\in GL(2,\R)$,  such that $g=\sum_{k=0}^n c_k H_k$, with $H_k$ Hermite function, forms
a   Gabor frame $\G(g,\Lambda)$. Besides they prove the existence of dual windows $\gamma$ that belong to the space $S^{1/2} _{1/2} (\R)$ (cf. \cite[Lemma 4.4]{GL09}).
This theory transfers to the $d$-dimensional case simply by taking  tensor products $g=g_1\otimes\cdots\otimes g_d\in S^{1/2} _{1/2} (\rd)$ of windows as above, which define a Gabor frame on the lattice $\Lambda_1\times\cdots\times\Lambda_d$ and   possess a dual window $\gamma=\gamma_1\otimes\cdots\otimes \gamma_d$ in the same space $\in S^{1/2} _{1/2} (\rd)$. \par
The Gabor super-frames allow the discretization of the kernel in \eqref{unobis2s}.
\begin{theorem}\label{equivdiscr-cont}  Let $\G(g,\Lambda)$ a Gabor super-frame for $\lrd$. Consider  $s\geq1/2$
 and a symbol  $\sigma\in\cC^\infty(\rdd)$. Then the following properties are equivalent:
\par
{(i)} There exists $\eps>0$ such that the estimate \eqref{unobis2s} holds.\par
{(ii)} There exists $\eps>0$ such that
\begin{equation}\label{unobis2discr} |\langle \sigma^w \pi(\mu)
g,\pi(\lambda)g\rangle|\lesssim  w_{s,-\eps}(\lambda-\mu),\qquad \forall\,
\lambda,\mu\in\Lambda.
\end{equation}
\end{theorem}
\begin{proof} Let us sketch $(ii)\Rightarrow (i)$. The pattern of \cite[Theorem 3.2]{charly06} can be adapted to this proof by using  a Gabor super-frame $\G(g,\Lambda)$, with a dual window $\gamma\in S^{1/2} _{1/2} (\rd)$ and the following property of the weights \eqref{pesiw} (\cite[Lemma 3.5]{CNR12}):
\begin{equation}\label{convpesi} (w_{s,-\eps}\ast w_{s,-\eps})(\lambda):=\sum_{\nu\in\Lambda}w_{s,-\eps}(\lambda-\nu)w_{s,-\eps}(\nu)\lesssim \left\{\begin{array}{ll}
w_{s,-\eps}(\lambda), & \,\mbox{for}\quad s\geq1\\
w_{s,\,-\eps 2^{-1/s}}(\lambda), &\,\mbox{for}\quad \frac12\leq s<1.
\end{array}
\right.
\end{equation}
\par
 Let $\mathcal{Q}$ be a symmetric relatively compact fundamental domain of the lattice $\Lambda\subset\rdd$.
 Given $w,z\in\rdd$, we can write them uniquely as $w=\lambda+u$, $z=\mu+u'$, for $\lambda,\mu\in\Lambda$ and $u,u'\in \mathcal{Q}$. Using the Gabor reproducing formula for the time-frequency shift $\pi(u)g\in S^{1/2} _{1/2} (\rd)$ we can write
\begin{equation*} \pi(u)g=\sum_{\nu\in\Lambda}\la\pi(u)g,\pi(\nu)\gamma\ra\pi(\nu) g.
\end{equation*}
Inserting the prior expansions in the assumption \eqref{unobis2discr},
\begin{align*}
&|\langle \sigma^w \pi(\mu+u')
g,\pi(\lambda+u)g\rangle|\notag\\
&\qquad\qquad\leq \sum_{\nu,\nu'\in\Lambda} |\langle \sigma^w \pi(\mu+\nu')g,\pi(\lambda+\nu)g\rangle|\,|\la \pi(u')g,\pi(\nu')\gamma\ra|\,|\la\pi(u)g,\pi(\nu)\gamma\ra|\notag\\
&\qquad\qquad\lesssim \sum_{\nu,\nu'\in\Lambda} m\left(\frac{\lambda+\mu+\nu+\nu'}2\right) e^{-\eps|\lambda+\nu-\mu-\nu'|^{\frac1s}}|V_\gamma g(\nu'-u')| |V_\gamma g(\nu-u)|\label{mag1}.
\end{align*}
Since the window functions $g,\gamma$ are both in $ S^{1/2} _{1/2} (\rd)$,  the  STFT $V_\gamma g$ is in $S^{1/2} _{1/2} (\rdd)$, due to \eqref{zimmermann1}. Thus there exists $h>0$ such that $ |V_\gamma g(z)|\lesssim e^{-h |z|^{2}}$, for every $z\in\rdd$.
Inserting this estimate in the previous majorization and using \eqref{convpesi} repeatedly,
\begin{align*}|\langle \sigma^w \pi(\mu+u')
g,\pi(\lambda+u)g\rangle|&\lesssim \sum_{\nu,\nu'\in\Lambda} e^{-\eps|\lambda+\nu-\mu-\nu'|^{\frac1s}}e^{-h |\nu|^{2}}e^{-h |\nu'|^{2}}\notag\\
&\lesssim \sum_{\nu,\nu'\in\Lambda} e^{-b (|\lambda+\nu-\mu-\nu'|^{\frac1s}+ |\nu|^{\frac1s}+|\nu'|^{\frac1s})}\\
&=w_{s,b}\ast w_{s,b}\ast w_{s,b})(\lambda-\mu)\\
&\leq e^{-\tilde{\eps}|\lambda-\mu|^{\frac1s}}
\end{align*}
for a suitable $ \tilde{\eps}>0$.\par
If $w,z\in\rdd$ and $w=\lambda+u$, $z=\mu+u'$, $\lambda,\mu\in\Lambda$, $u,u'\in\mathcal{Q}$, then $\lambda-\mu=w-z+u'-u$ and $u'-u\in \mathcal{Q}-\mathcal{Q}$, which is a relatively compact set, thus
\begin{equation}\label{stimaexp}e^{-\tilde{\eps}|\lambda-\mu|^{\frac1s}}\lesssim \sup_{u\in \mathcal{Q}-\mathcal{Q}}e^{-\tilde{\eps}|w-z+u|^{\frac1s}}\lesssim e^{-\tilde{\eps}|w-z|^{\frac1s}}.
\end{equation}
This gives the desired implication.
\end{proof}

\section{Sparsity of the Gabor matrix}\label{sec5}
 The operators $\sigma^w$ in Theorem \ref{caratterizz} enjoy a fundamental sparsity property. Indeed, let $\G(g,\Lambda)$ be a Gabor super-frame for $\lrd$. Then, as we saw,
 \begin{equation}\label{unobis2discrbis} |\langle \sigma^w \pi(\mu)
g,\pi(\lambda)g\rangle|\leq C e^{-\eps|\lambda-\mu|^{\frac1s}},\qquad \forall\,
\lambda,\mu\in\Lambda,
\end{equation}
with suitable constants $C>0$, $\eps>0$. This gives at once an exponential-type sparsity,
which is proved in \cite[Proposition 4.5]{CNR12} (we refer to
\cite{candes, guo-labate} for the more standard notion of super-polynomial sparsity) and recalled below.
\begin{proposition}
Let the Gabor matrix
$\langle \sigma^w \pi(\mu)
g,\pi(\lambda)g\rangle$ satisfy \eqref{unobis2discrbis}. Then it is sparse in the following sense.
Let $a$ be any column
or raw of the matrix, and let
$|a|_n$ be the $n$-largest
entry of the sequence $a$.
Then, $|a|_n$
satisfies
\[
|a|_n\leq C \displaystyle e^{-\epsilon n^{\frac1{2ds}}},\quad n\in\bN
\]
\end{proposition}
for some constants $C>0,\epsilon>0$.
\begin{proof}
By a discrete analog of Proposition \ref{equi} it suffices to prove that
\[
n^\alpha |a|_n\leq C^{\alpha+1}(\alpha!)^{2ds},\quad \alpha\in\mathbb{N}.
\]
On the other hand we have
\[
n^{\frac1p}\cdot|a|_n\leq
\|a\|_{\ell^p},
\]
for every $0<p\leq\infty$. Hence by \eqref{unobis2discrbis} and setting $p=1/\alpha$ we obtain
\[
n^\alpha |a|_n\leq \Big(\sum_{\lambda\in\Lambda} e^{-\epsilon p|\lambda-\mu|^{\frac1s}}\Big)^{\frac1p}
=
 \Big(\sum_{\lambda\in\Lambda} e^{-\epsilon p|\lambda|^{\frac1s}}\Big)^{\frac1p}.
\]
Let $\mathcal{Q}$ be a fundamental domain of the lattice $\Lambda$. Then if $x\in\lambda+\mathcal{Q}$, $\lambda\in\Lambda$, we have $|x|\leq |\lambda|+C_0$, therefore $|x|^{1/s}\leq C_1(|\lambda|^{1/s}+1)$. Hence
\begin{align*}
\sum_{\lambda\in\Lambda} e^{-\epsilon p|\lambda|^{\frac1s}}&\leq C_2 \int_{\rdd} e^{-\epsilon p|x|^{\frac1s}}\, dx=\int_{\mathbb{S}^{2d-1}}d\sigma\int_0^{+\infty} e^{-\eps p \rho^{\frac 1s}}\rho^{2d-1}d \rho\\
&= \frac {C_3s}{(\eps p)^{2ds}}\int_0^{+\infty} e^{-t} t^{2d s-1} dt= \frac{C_3s \Gamma(2ds)}{(\eps p)^{2ds}} =\frac{C_4}{p^{2ds}}
\end{align*}
Finally, by Stirling's formula,
\[
n^\alpha |a|_n\leq \frac{C_4^{1/p} }{p^{\frac{2ds}{p}}}\leq C_5^{\alpha+1}(\alpha!)^{2ds}.
\]
%\[
%n^\alpha |a|_n\leq C_2^{1/p} \Big(\int_{\rdd} e^{-\epsilon p|x|^{1/s}}\, dx\Big)^{1/p}=C_3^{1/p}(1/p)^{2ds/p}\leq C_4^{\alpha+1}(\alpha!)^{2ds},
%\]
%by Stirling's formula.
\end{proof}

\end{document}